\documentclass[12pt]{amsart}

\usepackage{amsmath}
\usepackage{enumitem}
\usepackage{amsthm}
\usepackage{amsfonts}
\usepackage{amssymb}
\usepackage{mathrsfs}
\usepackage{hyperref}
\usepackage[dvipsnames]{xcolor}
\usepackage[all]{hypcap}[2006/02/20]
\usepackage{fullpage}
\usepackage{esint}
\usepackage{graphicx}
\usepackage{enumitem}
\usepackage{tikz}

\DeclareMathAlphabet{\mathpzc}{OT1}{pzc}{m}{it}

\newcommand{\bb}[1]{\mathbb{#1}}

\newcommand{\pard}[2]{\frac{\partial #1}{\partial #2}}

\newcommand{\ho}{\left(\frac{d}{dt} -\Delta \right)}
\newcommand{\ddt}[1]{\frac{ d #1}{dt}}
\newcommand{\ip}[2]{\left \langle #1 , #2 \right\rangle}
\newcommand{\g}{\gamma}

\newcommand{\n}{\nabla}
\newcommand{\e}{\epsilon}

\renewcommand{\l}{\lambda}
\renewcommand{\d}{\delta}

\newcommand{\ra}{\rightarrow}
\newcommand{\ov}{\overline}
\newcommand{\un}[1]{{\underline{#1}}}

\newcommand{\T}{\Theta}

\newcommand{\on}[1]{{\operatorname{#1}}}

\begin{document}
\theoremstyle{plain}
\newtheorem{theorem}{Theorem}
\newtheorem{lemma}[theorem]{Lemma}
\newtheorem{claim}[theorem]{Claim}
\newtheorem{proposition}[theorem]{Proposition}
\newtheorem{cor}[theorem]{Corollary}
\theoremstyle{definition}
\newtheorem{defses}{Definition}
\newtheorem{example}{Example}
\newtheorem{assumption}{Assumption}
\theoremstyle{remark}
\newtheorem{remark}{Remark}

\title{A note on Alexandrov immersed mean curvature flow}
\author{Ben Lambert}
\email{b.s.lambert@leeds.ac.uk}
\address{ School of Mathematics, University of Leeds, Leeds, LS2 9JT, United Kingdom}
\author{Elena M\"ader-Baumdicker}
\email{maeder-baumdicker@mathematik.tu-darmstadt.de}
\address{Fachbereich Mathematik, Schlossgartenstr. 7, 64289 Darmstadt, Germany}
\begin{abstract}
We demonstrate that the property of being Alexandrov immersed is preserved along mean curvature flow. Furthermore, we demonstrate that mean curvature flow techniques for mean convex embedded flows such as noncollapsing and gradient estimates also hold in this setting. We also indicate the necessary modifications to the work of Brendle--Huisken to allow for mean curvature flow with surgery for the Alexandrov immersed, $2$-dimensional setting.
\end{abstract}
\maketitle

\section{Introduction}

One of the major breakthroughs in mean convex mean curvature flow is the notion of $\alpha$-noncollapsedness, which may be viewed as a ``quantitative embeddedness'' and has played a vital role many of the recent advances in Mean Curvature Flow (MCF). Originating in the work of White \cite{White} and Sheng and Wang \cite{ShengWang}, an elegant and influential proof of $\alpha$-noncollapsedness was given by Andrews \cite{AndrewsNoncollapsing}. This idea has led to numerous important results in the singularity theory of noncollapsed MCF, key amongst them are Haslhofer and Kleiner's gradient estimates \cite{HaslhoferKleiner} and Brendle's improved estimates on the inscribed radius \cite{BrendleInscribedRadius}. These developments contributed to Brendle and Huisken's notion of surgery for embedded mean convex MCF of surfaces \cite{BrendleHuisken}, extending the surgery results of Huisken--Sinestrari \cite{HuiskenSinestrariSurgery} (which only hold for dimensions $n\geq 3$). A second approach to mean curvature flow with surgery was described by Haslhofer and Kleiner \cite{HaslhoferKleinerSurgery} which also makes extensive use of noncollapsing quantities. The property of being $\alpha$-noncollapsed continues to play a vital role in many of the recent mean curvature flow developments, see for example \cite{AngenentDaskalopoulosSesum}\cite{BrendleChoi}\cite{ChoiHaslhoferHershkovits}\cite{ChoiHaslhoferHershkovitsWhite}.   

In this paper we demonstrate that in several of the above results we may weaken the initial embeddedness requirement to the property of being Alexandrov immersed. Indeed, we show that this property is preserved by the flow, we demonstrate a notion of $\alpha$-noncollapsedness for Alexandrov immersions, and we provide applications. 
We note that $\alpha$-noncollapsing methods have previously been applied to Alexandrov immersed surfaces in the case of minimal surfaces by Brendle \cite{BrendleAlex}, in an extension of his proof of the Lawson conjecture \cite{BrendleLawson}. We would also like to mention that Brendle and Naff found a way to do a localized version of noncollapsedness and get that certain ancient solutions of the MCF are noncollapsed, see \cite{BrendleNaff}.

We recall the definition of an Alexandrov immersed hypersurface. This property goes back to Alexandrov \cite{Alexandrov} in his work about closed surfaces of constant mean curvature in Euclidean space\footnote{Note that the definition of being Alexandrov immersed in \cite{Alexandrov} does not carry this name, of course. Also, the property of $G$ being an immersion is missing in this paper (there, only the expression \emph{smooth mapping} is used). But it is meant to be part of the definition as it is used in the proof.}.
\begin{defses}\label{def:AlexImmersed}
An immersion $X:M^n\ra \bb{R}^{n+1}$ is Alexandrov immersed if there exists an $(n+1)$-dimensional manifold $\ov{\Omega}$ with $\partial \ov{\Omega} = M^n$ and an immersion $G:\ov{\Omega} \ra \bb{R}^{n+1}$ such that $G|_{\partial \ov{\Omega}}$ also parametrises $\on{Im}(X)$. 
\end{defses}

In this paper we first prove:
\begin{center}
\emph{The property of being Alexandrov immersed is preserved under the flow, and there is a suitable notion of comparison solutions.}
\end{center}
For full statements, see Proposition \ref{prop:PresAlexImmersed} and Lemma \ref{lem:Comparison}. Next we demonstrate: 
\begin{center}
\emph{For mean convex initial data, there exits a natural notion of $\alpha$-noncollapsed  Alexandrov immersed mean curvature flow. Noncollapsedness is preserved under the flow.}
\end{center}
See Theorem \ref{thm:noncollapsing}, and also Propositions \ref{prop:innerviscosity} and \ref{prop:outerviscosity} for viscosity sub/super-solution equations for noncollapsing quantities. In particular, the above noncollapsing quantities still rule out ``grim reaper cross $\bb{R}$'' singularity models, which is one of the major stumbling blocks for immersed mean curvature flow with surgery in dimension $n=2$. Finally we verify:
\begin{center}
\emph{Mean convex Alexandrov immersed MCF satisfies gradient estimates and Brendle and Huisken's surgery may be extended to this setting}
\end{center}
These final results are extensions of existing results requiring only minor modifications, but they confirm the applicability of the Alexandrov immersed property as a natural extension to embedded property for MCF. We note that surgery in dimensions $n\geq 3$ for immersed surfaces is already known by the work of Huisken--Sinestrari \cite{HuiskenSinestrariSurgery}. We verify the above claims in Theorems \ref{thm:gradest}, \ref{thm:inscrradest} and \ref{thm:surgery}.\\[0.1mm]

\section*{Acknowledgments}

Both authors would like to thank Karsten Gro{\ss}e-Brauckmann for useful insights. The second author is supported by the DFG (MA 7559/1-2) and thanks the DFG for the support.

\section{Notation}
Let $M^n$ be an $n$-dimensional, closed manifold. Throughout this paper we will consider a smooth one-parameter family of immersions $X:M^n\times[0,T)\ra \bb{R}^{n+1}$ which satisfies mean curvature flow, that is,
\[\ddt{X} = -H\nu\ ,\]
 where we use the convention that $h_{ij} = - \langle \partial_i\partial_jX,\nu\rangle$ are the local coefficients of the second fundamental form $A$, and $H = \sum g^{ij}h_{ij}$ is the mean curvature. We will write the principal curvatures as 
 \[\kappa_{\on{min}}=\kappa_1 \leq \kappa_2 \leq \ldots \leq \kappa_n=\kappa_{\on{max}} \ .\]
At any given point we will take $e_1, \ldots , e_n$ to be an orthonormal basis of principal directions corresponding to the $\kappa_i$. Unless otherwise stated, we will assume that on the considered time interval 
\begin{equation}
|A|^2 \leq C_{A,0}\ ,\label{eq:Abound}
\end{equation}
which implies by standard methods (see e.g. \cite[Proposition 3.22]{EckerBook}) that 
\begin{equation}
|\n^p A|^2 \leq C_{A,p}\ .\label{eq:nAbound}
\end{equation}
At no point will the values of $C_{A,0}$ and $C_{A,p}$ be important, and so all theorems will carry through arbitrarily close to singularities. We will follow the usual mean curvature flow definitions, writing $M_t=\on{Im}(X(\cdot, t))$. 

We suppose that $\ov{\Omega}$ is a compact manifold with boundary such that $\partial \Omega$ is diffeomorphic to $M^n$. We take $M_0$ to be Alexandrov immersed with immersion $G_0:\ov\Omega\ra \bb{R}^{n+1}$ and choose $\nu(x,t)$ to be the unit vector at $X(x,t)$ which is continuous in space and time for which the pullback $G^*_0(\nu(x,0))$ points out of $\ov{\Omega}$. 

Along the flow, for some $\ov{T}\leq T$, we consider 
\[G:\ov{\Omega}\times[0,\ov{T})\ra \bb{R}^{n+1}\] 
to be a one parameter family of immersions so that for each $t\in[0, \ov{T})$, $X(\cdot,t)$ and $G(\cdot,t)$ satisfy Definition~\ref{def:AlexImmersed}. We assume that $G$ is smooth in space, but at this point assume no regularity or uniqueness in time (although later we prove that we may take $\ov{T}=T$ and $G$ may be locally chosen to be smooth). Different choices of $G$ are related by a time dependent diffeomorphism of $\ov\Omega$.

Throughout we will assume that $M_0$ is Alexandrov immersed at $t=0$, so such a $G$ exists at least for $t=0$.
\section{Preservation of being Alexandrov immersed}

For any given choice of $G$ as above, we define $\ov{g}_{\alpha\beta}(x,t)$ to be the pullback of the Euclidean metric on $\bb{R}^{n+1}$ to $\ov{\Omega}$ along $G(\cdot, t)$.  For $x\in \partial \Omega$ we define $\gamma_t(x,s)$ to be the unit speed geodesic with respect to $\ov{g}$ starting at $x$ and going into $\ov{\Omega}$ (for $s>0$), with starting direction $\ov{g}$ orthogonally to $\partial \Omega$ (i.e. the geodesic with initial velocity $G^*_t(-\nu)$).
Using this we define the injectivity radius to be
\[\on{inj} (t):=\sup\{\lambda>0 \ | \ \g_t: \partial \ov{\Omega} \times [0,\l) \text{ is injective}\}\ .\]
Note that the injectivity radius does not depend on the choice of $G$ as geodesics of a metric do not depend on the parametrisation.

It will be useful to have canonical charts of an Alexandrov immersion initially, and the following describes one way of constructing them.

\begin{lemma}[Canonical charts] Suppose that $M_0$ is compact and Alexandrov immersed by $G:\ov{\Omega}\ra \bb{R}^{n+1}$ such that the curvature bound \eqref{eq:Abound} holds. Then we may construct a finite atlas of canonical charts $\phi_i:W_i \ra S_i$ of $\ov{\Omega}$ where $W_i\subset \ov\Omega$ and $S_i\subset\bb{R}^{n+1}$ are simply connected open sets such that\label{lem:canonicalconstruction} on each subset $S_i$, $D\widetilde{G}=\on{Id}$ and the pullback metric is the Euclidean metric $\delta_{\alpha\beta}$.
\end{lemma}
\begin{proof}
We construct an explicit manifold $\widetilde{\Omega}$ diffeomorphic to $\ov{\Omega}$ (or equivalently, an atlas of charts of $\ov{\Omega}$) in the following way. 

Take the pullback metric w.r.t.\ $G$ on $\ov{\Omega}$, and consider balls $B_r(x)$ of radius $r$ centered at $x\in\ov{\Omega}$. For $r<(10C_{A,0})^{-1}$, we now show that $G$ restricted to $B_r(x)$ must be a diffeomorphism onto its image: 

The only thing to check is injectivity. Suppose that $G(p)=G(q)$ for $p,q\in B_r(x)$. In this case the $\ov{\Omega}$ distance between $p$ and $q$ is realised by a $C^1$ curve (which is $C^2$ almost everywhere) made up of segments of geodesics and curves in $\partial \Omega$. By definition this curve has length less than $r=(10C_{A,0})^{-1}$ and, almost everywhere, curvature less than $C_{A,0}$. Therefore $G(\gamma)$ cannot self intersect, a contradiction.

Taking a finite cover $\{B_r(x_i)\}_{1\leq i\leq N}$ of $\ov\Omega$ by these balls and defining $S_i:=G(B_r(x_i))$, we define charts by $G|_{B_r(x_i)}:B_r(x_i) \ra S_i$, and note that by definition $G$ is the identity in these coordinates.  

\end{proof}

The following simple differential topology Lemma holds for general smooth flows which satisfy \eqref{eq:Abound} and \eqref{eq:nAbound}.
\begin{lemma}
Suppose that $M_0$ is Alexandrov immersed, and for $t\in[0,T)$, $M_t$ is a smooth flow such that the curvature bounds \eqref{eq:Abound} and \eqref{eq:nAbound} hold and suppose that on $M_0$,
\begin{equation}
\on{inj}(0)>\e\ . \label{eq:injbound}
\end{equation}  
Then there exists a constant $\tau=\tau(\epsilon, C_{A,0}, C_{A,1})$ such that $M_t$ remains Alexandrov immersed for all $t\in[0,\tau)$. 
\end{lemma}
\begin{proof}
At time $t=0$ we work in the canonical coordinates as described in Lemma \ref{lem:canonicalconstruction} which we now denote $\widetilde{G}:\ov{\Omega}\ra\bb{R}^{n+1}$. We take $X(\cdot,0)=\widetilde{G}|_{\partial \ov{\Omega}}$, which we flow by MCF to get $X(\cdot, t)$. We write the pullback of the Euclidean metric with respect to $\widetilde{G}$ as $\widetilde{g}$, and quantities calculated with respect to this fixed metric will be denoted with a tilde. 

For $3\delta\leq \e$, we define 
\[\Omega_\d=\Omega\setminus\widetilde{\gamma}(\partial \ov{\Omega}\times[0,2\d])\] 
where $\widetilde{\gamma}$ are inward pointing geodesics as above. Let $\Pi:\ov{\Omega}\setminus \Omega_\delta\ra \Sigma$ be the smooth function defined by $\Pi(\tilde\gamma(x,\rho))=x$. 
Let  $\chi$ be a cutoff function with $\chi(v) = 0$ for $v\geq \d$ and $\chi(v)=1$ for $v\leq \frac \delta 2$. We define $\rho(x) = \widetilde{d}(x,\partial \Omega)$ where $\widetilde{d}(x,y)$ is the metric space distance arising from $\widetilde{g}$. For $y\in\ov{\Omega}$, we define $G$ to be $\widetilde{G}$, modified in a tubular region about $\partial \ov{\Omega}$:
\begin{equation}
G(y,t) = \chi\circ \rho(y) (X(\Pi(y),t)-\rho \nu(\Pi(y),t)) + (1-\chi\circ \rho(y)) \widetilde{G}(y,t)\label{eq:Gmod}.
\end{equation}
Here, we have abused notation to arbitrarily extend $\Pi$ to $\ov{\Omega}$. We write $\ov{g}_t$ for the pullback metric coming from $G(\cdot,t)$.

For $\delta$ small enough (depending on the curvature bound) at $t=0$ we have that for any unit vector $v$,
\[\ov{g}_0 \geq \frac 1 6 \on{Id}\]
and so this is a parametrisation. Indeed, using $\widetilde{G}$ to locally identify $\ov{\Omega}$ with $\bb{R}^{n+1}$ for $Y$ a Euclidean unit vector orthogonal to the line $\widetilde{\gamma}$, we have that, 
\[D\Pi(Y) = \frac{1}{1-\rho h(Y,Y)}Y, \qquad DG(\cdot,0)(Y)=\left[\chi\circ \rho(\frac{1}{1-\rho h(Y,Y)} -\rho h(Y,Y))+(1-\chi\circ\rho)\right]Y,\]
 where we have used that the part where $\chi\circ\rho$ is differentiated vanishes because $X(\Pi(y),0) - \rho\nu(\Pi(y),0) = \tilde G(y,0)$ in our canonical representation of $\tilde G$. 
For $Z$, a unit vector in direction $\pard{\tilde\gamma(x,s)}{s}$, we have that 
\[DG(\cdot,0)(Z)=Z.\]
By picking coordinates in line with the principal directions we have that  $\ov{g}$ is the identity matrix and the rest given by \[(\chi\circ\rho (\frac{1}{1-\rho \kappa_i}-\rho \kappa_i)+1-\chi\circ\rho)^2.\]
For $\delta<\min\{\frac{1}{2C_{A,0}}, \frac 1 3 \e\}$, this may be estimated by $1-\frac{5}{6}\chi\circ\rho \geq \frac 1 6$, see \cite[p.\ 354]{GilbargTrudinger} for similar calculations.

As the curvature of the flowing manifold is bounded, we may also crudely estimate that \[|D(X_t(\Pi(y),t)-X_0(\Pi(y))|_{\tilde{g}}<C_1(C_{A,0}) t\]
and
\[|D(\nu_t(\Pi(y),t)-\nu_0(\Pi(y))|_{\tilde{g}}<C_2(C_{A,0}, C_{A,1}) t\]
using standard means. As $X_t$ is the only part of $G_t$ which depends on time, we may estimate that for any $\widetilde{g}$ unit vector $v$
\[|\ov{g}_t(v,v)-\ov{g}_0(v,v)|\leq C_3(C_{A,0}, C_{A,1})t\ ,\]
and so
\[\ov{g}_t(v,v)\geq \frac 1 6-C_3t\ .\]
Therefore we see that there exists a time depending only on  $C_{A,0}, C_{A,1}$ and $\e$  such that $G$ remains a parametrisation and $M_t$ remains Alexandrov immersed.
\end{proof}

\begin{cor}
If $M_0$ is Alexandrov immersed and $M_t$ is a smooth flow satisfying the estimate \eqref{eq:Abound} and \eqref{eq:nAbound} then if $M_t$ ceases to be Alexandrov immersed for the first time at time $T$ then as $t\ra T$, $\on{inj}(t)\ra 0$. 
\end{cor}

\begin{proposition}\label{prop:PresAlexImmersed}
Given any compact Alexandrov immersed $M_0$, mean curvature flow remains Alexandrov immersed until the first singular time.
\end{proposition}
\begin{proof}
We prove that the flow remains Alexandrov immersed while \eqref{eq:Abound} and \eqref{eq:nAbound} hold.

From the previous Corollary, for the flow to cease to be Alexandrov immersed at time $T$ we must have that $\on{inj}_t\ra 0$ as $t\ra T$.

By compactness, at any given time $t$, we may find $p(t), q(t)$ which realise $\on{inj}(t)$. That is, a point $p(t)$ and a point $q(t)$ such that $\gamma_t(p(t), \on{inj}(t))=\gamma_t(q(t), \on{inj}(t))$. In fact, there is an unbroken unit speed geodesic $\gamma$ starting at $p(t)$ and ending at $q(t)$ of length $2\on{inj}(t)$ and $\ip{\nu(p,t)}{X(p,t)-X(q,t)}=-2\on{inj}(t)=\ip{\nu(q,t)}{X(q,t)-X(p,t)}$ (as otherwise we may find a geodesic which contradicts the definition of $\on{inj}(t)$). 

Taking a sequence of times $t_i\ra T$. Then for $(p_i,q_i):=(p(t_i),q(t_i))$, as above, there exists a subsequence which converges to some $(p,q)\in M^n\times M^n$ as $i \ra \infty$, where $p$ and $q$ are distinct. We consider $M_t\cap B_r(X(p,t))$ where $r<\frac 1 2 C_{A,0}^{-1}$. For $t$ close enough to $T$, the connected component of $M_t\cap B_r(X(p,T))$ which contains $X(p,T)$ may be written as a graph in graph direction $\nu(p,T)$. Furthermore, an open region about $X(q,t)$ may also be written as a graph in direction $\nu(p,T)$ (in fact, $\nu(q,T)=-\nu(p,T)$). These graphs are initially disjoint and must remain so until time $T$ (due to the curvature bound and the nonzero injectivity). However, this contradicts the strong maximum principle for the difference between the two graph functions, and so we must have that $\on{inj}_t>0$ for all time. 
\end{proof}

The above shows the following:
\begin{cor}
A compact flowing manifold with bounds on the curvature \eqref{eq:Abound} and \eqref{eq:nAbound}  may only loose the property of being Alexandrov immersed at time $T$ if there exist points $p, q(t)\in M^n$ so that $|X(p,t)-X(q(t),t)|$ goes to zero with $\ip{\nu(p,t)}{\nu(q(t),t)}=-1$ and where \\ $\ip{\nu(p)}{X(q(t),t)-X(p,t)}<0$ and $\ip{\nu(q,t)}{X(p,t)-X(q(t),t)}<0$ for $T-\delta<t<T$.
\end{cor}

We now consider the idea of a comparison solution for Alexandrov immersions. \begin{defses} Suppose that we have a larger compact $n+1$ dimensional manifold with smooth boundary $\ov{\Psi}$ such that $\ov{\Omega}\subset\ov{\Psi}$. Suppose that \label{def:AlexComp}
\begin{enumerate}
\item At time $t=0$, $G(\cdot,0)$ may be extended to $\ov{G}:\ov{\Psi}\ra \bb{R}^{n+1}$. \label{barrier1}
\item There is a parametrisation $\ov{G}:\ov{\Psi}\times[0,T) \ra \bb{R}^{n+1}$. We write the image of $\partial \ov \Psi$ under $G$ to be $N_t$, and we assume that this is a smoothly varying manifold. We write its outward unit normal as $\mu$ and write a local parametrisation of $N_t$ as $Y$. \label{barrier2}
\item We assume that $N_t$ satisfies a barrier condition, namely that
\[\ip{\ddt{Y}}{\mu(x,t)}\geq -H^N.\] \label{barrier3}
\end{enumerate}
Then $N_t$ will be called an \emph{Alexandrov comparison solution} for $M_t$.
\end{defses}
We note that $\partial \Omega$ is embedded in $\ov\Psi$ at time $t=0$ and we may view the pullback of $M_t$ as an initially embedded mean curvature flow with respect to $\ov{g}(\cdot,t)$, starting from $\partial \Omega$. The following proposition is therefore a minor embellishment on the standard comparison Lemma for mean curvature flow.

\begin{lemma}[Comparison solutions]\label{lem:Comparison}
Suppose that Definition \ref{def:AlexComp}\ref{barrier1}), \ref{def:AlexComp}\ref{barrier2}), \ref{def:AlexComp}\ref{barrier3}) above hold. Then, for $\ov{g}$ the pullback of the Euclidean metric, writing $d=\on{dist}_{\ov{g}}(\partial \ov\Omega, \partial \ov\Psi)$, $d$ is non-decreasing. $\ov{G}^*M_t$ remains embedded in $\ov\Psi$ and we may therefore identify $\ov\Omega$ with the closure of a moving domain in $\ov\Psi$.
\end{lemma}
\begin{proof}
As $\partial \Omega$ is embedded inside $ \ov{\Psi}$, this is essentially the standard proof, see \cite[Theorem 2.2., p28]{Mantegazza}. If $\partial \ov{\Omega}$ hits $\partial \ov{\Psi}$, then we again get a contradiction to the strong maximum principle for local graphs of mean curvature flow. The only modifications here is that we have allowed barriers - it is quick to check that these do not affect the inequalities which lead to the required contradiction. Similarly, by considering a time of first intersection of $M_t$ with itself, we see that $M_t$ remains embedded in $\ov{\Psi}$ as otherwise we again contradict the strong maximum principle.
\end{proof}

\section{Preservation of noncollapsedness}
We take $\ov{g}$ to be the time dependent metric on $\ov \Omega$ given by pulling back the Euclidean metric.

\begin{defses}\label{def:innernoncollapsed}
An Alexandrov immersed, strictly mean convex manifold is inner $\alpha$-non\-collapsed if at every $x\in M^n=\partial \ov{\Omega}$ there is a closed geodesic ball (w.r.t.\ $\ov{g}$) $B^x_{r}\subseteq \Omega$ of radius $r=\frac \alpha {H(x)}$ with $x\in \partial B^x_{r}$.
\end{defses}

Clearly the inside of the domain plays an important role here, and this is not available in the outer noncollapsing case. We instead use an Alexandrov outer comparison solution, as in Definition \ref{def:AlexComp}. Essentially this replaces $\bb{R}^n\setminus \Omega$ in the embedded case. We define the closed set $\ov \chi_t = \ov{\Psi}_t\setminus \overset{\circ}{\ov{\Omega}}_t$ and, as previously, we may define pull back the Euclidean metric to get a time dependent metric $\ov{g}$ on $\Psi\setminus \Omega$.
\begin{defses}\label{def:outernoncollapsed}
An Alexandrov immersed, strictly mean convex manifold is $\alpha$-non\-collapsed (with respect to $N_t$) if at every point $x\in M^n$ there is a closed geodesic ball (w.r.t.\ $\ov{g}$) $B^x_r\subseteq \chi$ of radius $r=\frac{\alpha}{H(x)}$ with $x\in \partial B_r^x$. 
\end{defses}

\begin{theorem}\label{thm:noncollapsing}
Suppose that $M_0$ is compact, Alexandrov immersed and strictly mean convex. Then there exists an $\alpha$ such that $M_t$ is inner $\alpha$-non\-collapsed and there exists an Alexandrov comparison solution $N_t$ such that $M_t$ is outer $\alpha$-non\-collapsed with respect to $N_t$.
\end{theorem}
\begin{proof}
Inner noncollapsing follows from Proposition \ref{prop:innerviscosity} and Corollary \ref{cor:innernoncollapsed} below. 

For outer noncollapsing, we first require a suitable comparison solution which we now construct. For $X_0=G|_{\partial \ov{\Omega}}$, we consider $Y:\partial \ov{\Omega} \times(-\delta, \delta]\ra \bb{R}^{n+1}$ given by 
\[Y(x,\rho) = X_0(x)+\rho \nu(x)\ .\]
For $\delta$ sufficiently small, for $-\delta<s\leq 0$ we identify $(x,s)$ with $\gamma_0(x,-s)\in \ov\Omega$ (where $\gamma$ is the geodesic mapping from the beginning of Section 2). This defines a smooth manifold with boundary $\ov{\Psi}$. We define the extension $\ov{G}:\ov{\Psi} \ra \bb{R}^{n+1}$ using the construction in equation \eqref{eq:Gmod} by $\ov{G}(y) = G(y,0)$ for $y\in \ov{\Omega}$ and $\ov{G}(x,\rho)=Y$ for $(x,\rho) \in \partial \ov\Omega \times (0,\delta]$. Furthermore, as $H>0$ on $M_0$, by restricting $\delta$ further we may assume that $N_0$ has positive mean curvature. As a result, by choosing $N_t=N_0$, the comparison equation of Definition \ref{def:AlexComp}(\ref{barrier3}) is satisfied. We therefore have a suitable Alexandrov comparison solution.

Outer noncollapsing with respect to $N_t$ now follows from Proposition \ref{prop:outerviscosity} and Corollary \ref{cor:outernoncollapsed} below.
\end{proof}
\begin{remark}
As with embedded MCF, if we start with $M_0$ only weakly mean convex, the strong maximum principle implies that the flow immediately becomes strictly mean convex, and, by Proposition \ref{prop:PresAlexImmersed}, remains Alexandrov immersed and so we can apply Theorem \ref{thm:noncollapsing}. 
\end{remark}
\begin{remark}
An alternative method to get a comparison solution in the above proof would be to run the flow for some small time $0<\tau$ strictly before the first singular time and then use $M_0$ as the required comparison solution for outer noncollapsing (or even $M_{t-\tau}$ for a dynamic comparison solution).
\end{remark}

Rather than attacking $\alpha$-noncollapsedness directly as in the work of Andrews \cite{AndrewsNoncollapsing}, we instead follow ideas of Andrews--Langford--McCoy \cite{AndrewsLangfordMcCoy} and consider the ``interior and exterior sphere curvature'', that is functions $\ov{Z}(x,t)$, $\un{Z}(x,t)$ which are given by the principal curvature of the largest round sphere tangent to $x$ which is inside $\ov{\Omega}$ or $\ov{\chi}$ respectively. These functions are continuous but not smooth in general, and the aim is to show that these are a viscosity sub/super solutions of the evolution equation for $H$.

As in \cite{AndrewsLangfordMcCoy}  we require the double point function $Z:M^n\times M^n\times[0,T)\ra \bb{R} \cup \infty$ given by
\begin{equation}Z(x,y,t) = \frac{2\ip{X(x,t)-X(y,t)}{\nu(x,t)}}{\|X(x,t)-X(y,t)\|^2}\ .\label{eq:Zdef}
\end{equation}
A short calculation implies that this quantity is the principal curvature of the unique sphere tangent to $\partial M$ at $x$ going through $y$ with a positive sign if the sphere is on the opposite side of $M_t$ as the normal and a negative sign otherwise.  Note that at self-intersections of $X$, we get $Z=\infty$. In order to avoid that while maximising $Z$ over $y$ we work on the ``visible set''. We make this precise in the next section.  

\subsection{Inner noncollapsing via viscosity solutions}
To avoid overlapping regions when maximising $Z$ in \eqref{eq:Zdef}, we consider all points which are visible from a point $x\in M^n$ inside $\ov \Omega$. Specifically, we define the \emph{visible set} for $x\in M^n$ at time $t$ to be
\[V(x,t):=\{y\in \ov{\Omega}: \exists \text{ a geodesic (w.r.t.\ $\ov{g}$), }\gamma:[0,1]\ra \ov{\Omega} \text{, with } \gamma(0)=x, \ \gamma(1)=y \}\ .\]
We may identify $V(x,t)$ with $G(V(x,t))$ - by definition $V(x,t)$ is isometric to a starshaped region of Euclidean space. Furthermore, by convexity, any curvature ball $B^x_r$ at $x$ must be contained in $V(x,t)$.

\begin{lemma}\label{lem:noselfint}
 Let $X:\partial\ov{\Omega}\times [0,T)\to\mathbb{R}^{n+1}$ be Alexandrov immersions and $x\in \partial \ov{\Omega}$. Then there are no points $y\in V(x,t)\cap \partial\ov{\Omega}$ with $y\not =x$ such that $X(x,t)=X(y,t)$, i.e.\ there are no self-intersections of $X$ restricted to the visible set.
\end{lemma}

\begin{proof}
 Using the canonical charts we see that geodesics $\gamma$ with respect to $\bar g$ are straight lines in $\mathbb{R}^{n+1}$. On the other hand, by using the definition of $\bar g$ as the pullback metric of the Euclidean metric on the domain $\ov{\Omega}$, we see that we need to have a geodesic with positive length in order to connect a point $y\in\partial\ov{\Omega}$ with $x\not =y$. But by going back to the image we see that $X(y,t)$ cannot be connected to $X(x,t)=X(y,t)$ via a straight line of positive length. So $X(x,t)=X(y,t)$ is not possible for $y\not =x$ on the visible set.
\end{proof}

We may now define the \emph{inscribed interior sphere curvature at $(x,t)$} to be
\[\ov{Z}(x,t):=\sup\{Z(x,y,t):y\in \partial V(x,t)\cap \partial \ov{\Omega}, \  y\neq x\}\ .\]

Our aim below is to show that we can follow Andrew--Langford--McCoy's analysis \cite{AndrewsLangfordMcCoy} restricting to the visible region. In particular, this will require checking that there are no problems with ``boundary points'' in the maximum principle.

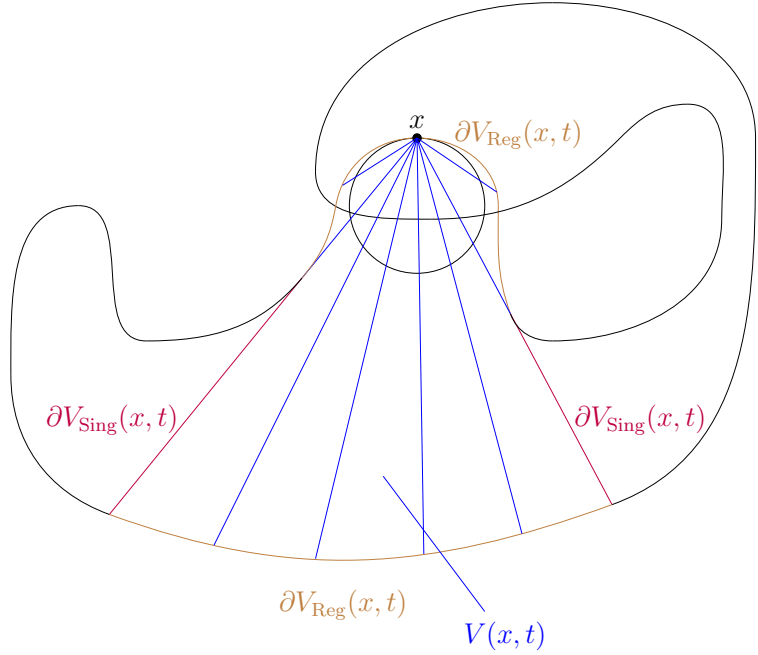
\begin{figure}
\begin{center}
\scalebox{0.9}{ 
   \begin{tikzpicture}
\draw (2,1) circle (1cm);
\draw[out=0, in=90, color=brown] (2,2) to (3.2,1);
\draw[out=-90, in=110, color=brown] (3.2,1) to (3.38,-0.6);
\draw[color=purple](3.38,-0.6) to (4.88,-3.42);
\draw[color=blue] (2,2) to (3.38,-0.6);
 \draw[out=-70, in=180] (3.38,-0.6) to (4,-1);
\draw[out=0, in=-90] (4,-1) to (6.5,0.8);
\draw[out=90, in=0] (6.5,0.8) to (6,2.5);
\draw[out=180, in=45] (6,2.5) to (5,2);
\draw[out=-135, in=0] (5,2) to (2.2,0.8);
\draw[out=180, in=-90] (2.2,0.8) to (0.5,1.5);
\draw[out=90, in=180] (0.5,1.5) to (4,4);
\draw[out=0, in=90] (4,4) to (7,2);
\draw[out=-90, in=20] (7,2) to (4.88,-3.42);
\draw[out=-160, in=-20, color=brown] (4.88,-3.42) to (-2.545,-3.565);

\draw[out=160, in=-90] (-2.545,-3.565) to (-4,-1.5);
\draw[color=purple] (0.3,-0.07) to (-2.545,-3.565);
\draw[color=blue] (0.3,-0.07) to (2,2);
\draw[out=90, in=180] (-4,-1.5) to (-3,1);
\draw[out=0, in=180] (-3,1) to (-2,-1);

\draw[out=0, in=-130] (-2,-1) to (0.3,-0.07);
\draw[out=50, in=-100, color=brown] (0.3,-0.07) to (0.8,1);
\draw[out=80, in=180, color=brown] (0.8,1) to (2,2);
\coordinate (a) at (2,2);
	  \fill (a) circle (2pt);
   \coordinate[label= above:$x$] (a) at (2,2);
      \coordinate[label= below:{{\color{purple}$\partial V_{\text{Sing}}(x,t)$}}] (b) at (-2.5,-1.8);
         \coordinate[label= below:{{\color{purple}$\partial V_{\text{Sing}}(x,t)$}}] (e) at (5.3,-1.8);
       \coordinate[label= left:{{\color{brown}$\partial V_{\text{Reg}}(x,t)$}}] (c) at (2,-4.9);
          \coordinate[label= above:{{\color{brown}$\partial V_{\text{Reg}}(x,t)$}}] (f) at (3.5,1.7);
       \draw[color=blue] (2,2) to (0.9,1.3);
          \draw[color=blue] (2,2) to (3.18,1.2);
          \draw[color=blue] (2,2) to (-1,-4.02);
          \draw[color=blue] (2,2) to (0.5,-4.22);
          \draw[color=blue] (2,2) to (2.1,-4.15);
          \draw[color=blue] (2,2) to (3.55,-3.85);
          \draw[color=blue] (1.5,-3) to (3,-5);
           \coordinate[label=below:{{\color{blue}$ V(x,t)$}}] (d) at (3.3,-5);
  \end{tikzpicture}

}

\caption{The visible set of a point $x\in M_t$.}
\end{center}
\end{figure}

\begin{lemma} \label{lem:ZtoIntSphere}
There exists an interior sphere with principal curvature $\kappa$ contained in $\ov{\Omega}$ tangent to $x\in M_t$ iff
\[ \sup_{y\in \partial V(x,t)\setminus\{x\}} Z(x,y,t)\leq \kappa\ .\]
\end{lemma}
\begin{proof}
Any sphere in $\ov\Omega$ tangent to $x$ is contained in $V(x,t)$ by convexity. Therefore, there is a sphere of radius $r=\kappa^{-1}$ in $\ov{\Omega}$, tangent to $M_t$ at $x$ iff for all $y\in \partial V(x,t)$, 
\[\|X(y,t)-(X(x,t)-r\nu(x,t))\|^2 - r^2\geq 0\ ,\]
or equivalently
\[\|X(y,t)-X(x,t)\|^2+2r\ip{\nu(x,t)}{X(y,t)-X(x,t)}\geq 0\ .\]

By Lemma~\ref{lem:noselfint} we do not have self-intersections on the visible set, so for $x \neq y$ this is equivalent to 
\[Z(x,y,t) = \frac{2\ip{X(x,t)-X(y,t)}{\nu(x,t)}}{\|X(x,t)-X(y,t)\|^2}\leq \frac 1 r=\kappa\ .\]
In particular if there exists an inscribed curvature ball, $\sup_{y\in V(x,t)\setminus\{x\}} Z(x,y,t)\leq \kappa$.

Similarly if $\sup_{y\in V(x,t)\setminus\{x\}} Z(x,y,t)\leq \kappa$ then we have the required inequality for $x\neq y$, while for $x=y$ the required inequality holds trivially. Therefore there exists an inscribed ball.
\end{proof}

We begin by splitting $\partial V(x,t)$ into two sets. For $x\in M^n$ and $y\in V(x,t)$, let $\gamma_{x,y}:[0,1]\ra V(x,t)$ be the unit speed geodesic starting at $x$ and going to $y=\gamma_{x,y}(1)$. Let
\[\partial V_\text{Reg}(x,t) := \{y\in \partial V(x,t)\cap\partial\Omega: \text{$\gamma_{x,y}$ hits $\partial \Omega$ for the first time at $y$ and } \gamma'_{x,y}(1)\notin T\partial \Omega\}\ , \]
and
\[\partial V_\text{Sing}(x,t) := \partial V(x,t)\setminus \partial V_\text{Reg}(x,t)\ . \]
We note that (by inverse function theorem), $\partial V_\text{Reg}(x,t)$ is open. 

\begin{lemma}\label{lem:Regular}
Suppose that $x,y\in M^n$, $x\neq y$ are such that
\[\max_{ y'\in \partial V(x,t)\cap\partial\Omega}Z(x,y',t)=Z(x,y,t)\geq 0\ .\]
Then $y\in \partial V_\text{Reg}(x,t)$ or $Z(x,y,t)=0$.
\end{lemma}
\begin{proof}
We may identify $V(x,t)$ with its image under $G$ and work in $\bb{R}^{n+1}$. Suppose not, then (wlog) we may assume that $y\in \partial V_\text{Sing}(x,t)$ is a point such that the line from $x$ to $y$ does not intersect any points of $\partial \Omega$ (otherwise we simply observe that if $Z(x,y,t)\geq 0$, by passing to points on the line closer to $x$ we increase $Z(x,y,t)$) and such that $Z(x,y,t)>0$. In particular, $\ip{X(x,t)-X(y,t)}{\nu(x,t)}>0$.  We know that by definition of $\partial V_\text{Sing}(x,t)$ that $X(x)-X(y)\in T_{X(y)}M_t$.

But then, treating $x$ as a constant we have that 
\begin{flalign}
D{}^y_{X(x,t)-X(y,t)}Z(x,y,t) &= Z(x,y,t)>0\ .\label{eq:OK}
\end{flalign}
We take a unit speed geodesic, $\widetilde{\gamma}$ from $y$ to $x$ is in $\ov{\Omega}$, and in a neighbourhood of $y$ we may project this line in direction $\nu(y)$ to $M_t$ to get a smooth curve $\tilde{a}(s)$ in $M_t$. By compactness, we may see that for small $s$, $a(s)$ stays in $\partial V(x,t)$ (otherwise we either contradict that there are no other points on the line from $x$ to $y$ or that $\ip{X(y,t)-X(x,t)}{\nu(x)}<0$). Using \eqref{eq:OK}, for small $s$ we have that  $Z(x,a(s),t)>Z(x,y,t)$, a contradiction to the maximality of $Z(x,y,t)$.
\end{proof}

\begin{lemma}\label{lem:continuity}
$\ov{Z}(x,t)$ is continuous in time and space while and $\ov{Z}(x,t) \geq 0$ and the flow is smooth, that is, \eqref{eq:Abound} and \eqref{eq:nAbound} hold. 
\end{lemma}
\begin{proof} 
By Proposition \ref{prop:PresAlexImmersed}, the flow remains Alexandrov immersed, and so prior to the singular time, $\ov{Z}$ is bounded at every point.

We consider a spacetime sequence $(x_k, t_k)\ra (x_\infty, t_\infty)$ and we aim to show that \\
$\lim_{k\ra\infty}\ov{Z}(x_k, t_k)=\ov{Z}(x_\infty, t_\infty)=:\ov{Z}_\infty$. 

Suppose that $\liminf_{k\ra\infty}\ov{Z}(x_k, t_k) = \ov{Z}_\infty - 2\e$. Then there exists a subsequence $(x_{k(i)},t_{k(i)})$ such that $\lim\ov{Z}(x_{k(i)}, t_{k(i)})= \ov{Z}_\infty - 2\e$. In particular, we may find an interior ball of radius $\frac{1}{\ov{Z}_\infty - \e}$, tangent to $M_{t(i)}$ at $x_{k(i)}$, and otherwise disjoint from $M_{t_{k(i)}}$ by Lemma \ref{lem:ZtoIntSphere}. In this case we also know that $\ov{Z}_\infty$ is not $\kappa_{\on{max}}(x_\infty, t_{\infty})$ by continuity of $\kappa_{\on{max}}$, and so $\ov{Z}_\infty$ must be realised by $\ov{Z}_\infty=Z(x_\infty, y_\infty, t_\infty)$ for some $y_\infty\in M_{ t_\infty}$. In particular, $y_\infty$ is in the boundary of an inscribed ball of radius $\frac{1}{\ov{Z}_\infty}<\frac{1}{\ov{Z}_\infty-\e}$. However, for $i$ large enough, this contradicts the bound on the curvature as the manifold must move in at infinite speed to reach $y_\infty$, a contradiction. Therefore $\liminf_{i\ra\infty}\ov{Z}(x_k, t_k) \geq \ov{Z}_\infty$.

Now suppose that $\limsup_{k\ra\infty}\ov{Z}(x_k, t_k) = \ov{Z}_\infty + 2\e$. Then there exists a subsequence $(x_{k(i)},t_{k(i)})$ such that $\lim\ov{Z}(x_{k(i)}, t_{k(i)})= \ov{Z}_\infty + 2\e$. In particular, we may find points $y_i\in M_{t_{k(i)}}$ inside an interior ball of radius $\frac{1}{\ov{Z}_\infty + \e}$ tangent to $x_{k(i)}$. Taking a further subsequence, the $y_i$ converge to $y\in M_{t_{\infty}}$ which is inside the inscribed ball of radius $\frac{1}{\ov{Z}_\infty+\e}$ tangent to $x_\infty$. But this contradicts the definition of $\ov{Z}(x_\infty, t_\infty)$. Therefore $\limsup_{k\ra \infty} \ov{Z}(x_k, t_k)\leq \ov{Z}_\infty$. The claim now follows.

\end{proof}

 We recall that a continuous function $f:M^n\times[0,T)\ra \bb{R}$ is a viscosity subsolution of $\ho f = F(x,t,f,\n f)$ if for every $(x_0,t_0)\in M^n\times[0,T)$ and every $C^2$ function $\phi$ on $M\times [0,T)$ such that $\phi(x_0,t_0)=f(x_0,t_0)$ and $\phi\geq f$ for $x$ in a neighbourhood of $x_0$ and $t\leq t_0$, then $\ho \phi \leq F(x,t,\phi,\n \phi)$ at $(x_0, t_0)$. In \cite[Definition 9.5, p287]{AndrewsChowGuentherLangford} such solutions are alternatively called barrier subsolutions.

In the proposition below, we work in orthonormal coordinates at a point such that $g_{ij}=\delta_{ij}$ and $h_{ij}=\kappa_i \delta_{ij}$.
\begin{proposition}\label{prop:innerviscosity}
While $\ov{Z}\geq 0$, $\ov{Z}$ satisfies 
\[\ho \ov{Z}\leq |A|^2 \ov{Z}-2\sum_{\kappa_i<\ov{Z}} \frac{(\n_{e_i} \ov{Z})^2}{\ov{Z}-\kappa_i}\]
in the viscosity sense.
\end{proposition}
\begin{proof}
This proof is only a small modification of \cite[Proof of Theorem 2]{AndrewsLangfordMcCoy}. For the convenience of the reader we provide the proof here. 

As $V(x,t)$ is compact so is $V(x,t)\cap M_t$ and so we have that either $\ov{Z}(x,t)$ is realised by $Z(x,y,t)$ for some $y\in (\partial V(x,t)\cap M_t)\setminus B_\e(x)$ (for some $\e>0$) or there exists a sequence $y_i\in V(x,t)$, $y_i\ra x$ such that $\ov{Z}(x,t)=\lim_{i\ra\infty} Z(x,y_i,t)$. In the latter case, applying Taylor's theorem, we have that 
\[\ov{Z} = h|_{(x,t)} (v,v)\]
for some unit vector $v\in T_xM_t$. We note that for unit vectors $v\in T_p M$ such that $ h|_{(x,t)} (v,v)>0$, we have that $\operatorname{exp}_{x}(\lambda v)\in V(x,t)$ for $\lambda$ sufficiently small. Writing $\kappa_\text{max}(x,t)$ for the maximum principal curvature at $x\in M_t$, we either have that $\ov{Z}(x,t)=\kappa_{\text{max}}(x,t)$ or there exists $x\neq y\in V(x,t)$ such that $\ov{Z}(x,t)=Z(x,y,t)$, or both. 

Pick $t_0\in[0,T)$ and $x_0\in M_{t_{0}}$. We now suppose that we have a $C^2$ function $\phi(x,t)$ such that in a neighbourhood of $(x_0,t_0)$ for $t\leq t_0$, $\ov{Z}(x,t)\leq \phi(x,t)$ with equality at $(x_0,t_0)$. 

As a supremum is taken we know that for $(x,t)$ as above, for any unit vectors $v\in T_xM_t$, 
\begin{equation} h(v,v)|_{(x,t)}\leq \ov{Z}(x,t)\leq \phi(x,t)\label{eq:nearx}\end{equation}
and for all $y\in M_t\cap \partial V(x,t)$
\begin{equation}Z(x,y,t)\leq \ov{Z}(x,t)\leq \phi(x,t)\ .\label{eq:farfromx}\end{equation}

Suppose that, at $\ov{Z}(x_0,t_0)=\kappa_{\on{max}}(x_0,t_0)=h(v_0,v_0)|_{(x_0,t_0)}$ for some unit vector $v_0\in T_{x_0}M_{t_0}$. In this case we may immediately apply the known standard viscosity subsolution equation \cite[Proposition 12.9, equation (12.17)]{AndrewsChowGuentherLangford} for $\kappa_{\on{max}}$ (which requires no assumptions on embeddedness) to see that $\phi$ satisfies 
\[\ho \phi \leq |A|^2\phi -2\sum_{\kappa_i<\phi} \frac{(\n_{e_i} \phi)^2}{\phi-\kappa_i}\]

Suppose now that $\kappa_{\text{max}}<\ov{Z}(x_0,t_0)=Z(x_0,y_0,t_0)$ for some $x_0\neq y_0$. If $\ov{Z}=0$, then by assumption we are at a minimum and the required equation follows immediately from properties of $\phi$ at a space-time minimum. Otherwise, applying Lemma \ref{lem:Regular} we know that $y_0\in \partial V_\text{Reg}(x_0,t_0)$ and inverse function theorem implies that we also have that there exists $\e, \delta>0$ such that if $x\in B_\e(x_0)$ then all $y\in B_{\delta}(y_0)\cap M_t$ satisfy that $y\in \partial V_\text{Reg}(x,t)$ for all $t_0-\e<t\leq t_0$. As a result, for such $x,y,t$, equation \eqref{eq:farfromx} holds, and furthermore, as the above sets are open, we may differentiate this inequality in time and space directions at $(x_0,y_0,t_0)$. 

Following \cite{AndrewsLangfordMcCoy}, by writing  $d=|X(x,t)-X(y,t)|$, $w=d^{-1}(X(x,t)-X(y,t))$ we have that at $x_0,y_0,t_0$, 
\begin{align*}
0=\pard{}{y^i}(\phi-Z) &= \frac{2}{d^2}\ip{\partial_i^y}{\nu(x_0,t_0) - d Z w}\\
0=\pard{}{x^i}(\phi-Z)&= \pard{\phi}{x^i} - \frac 2 d (h^{xp}_i\ip{w}{\partial^x_p}-Z\ip{w}{\partial^x_i})
\end{align*}
We note that the first of these implies that $\nu(x_0,t_0) - d Z w$ is orthogonal to $T_yM_t$, and a short calculation implies that this is a unit vector. Indeed, as the inscribed ball has center $X(x,t)-Z^{-1}\nu(x,t)=X(y,t)-Z^{-1}\nu(y,t)$, we may rearrange to get that 
\begin{equation}
\nu(y,t) = \nu(x,t)-Zdw\ .\label{eq:nuy}
\end{equation}  
Using this and the above identities we have that at $x_0,y_0,t_0$
\begin{align*}
\frac{\partial^2}{\partial y^i\partial y^j}(\phi-Z)&=\frac 2{d^2}(Z\delta_{ij}-h^y_{ij})\\
\frac{\partial^2}{\partial x^j\partial y^i}(\phi-Z)&=-\frac{2}{d^2}\left(Z\delta_j^p-h_j^{xp}\right)\ip{\partial^y_i}{\partial^x_p}- \frac 2 d \pard{\phi}{x^j}\ip{w}{\partial^y_i}\\
\frac{\partial^2}{\partial x^i\partial x^j}(\phi-Z)&=\frac{2}{d^2}(Z\delta_{ij}-h^x_{ij}) + Zh^x_{jp}\delta^{pq}h^x_{qi} - \frac 2 d\n_p h^x_{ij}\delta^{pq}\ip{w}{\partial^x_q}-Z^2 h^x_{ij} \\
&\qquad+ \frac 2 d \pard{\phi}{x^j}\ip{w}{\partial^x_i}+ \frac 2 d \pard{\phi}{x^i}\ip{w}{\partial^x_j}+\frac{\partial^2 \phi}{\partial x^i \partial x^j}\\
\pard{}{t}(\phi-Z) &=\pard{\phi}{t}+\frac{2H(x)}{d^2} - \frac{2H(y)}{d^2} - \frac{2}{d}\ip{w}{\n H(x)}-Z^2H(x)
\end{align*}
and so as $(x_0,y_0,t_0)$ is a minimum of $\phi - Z$, we have that at that point (after some substitution),
\begin{flalign*}
0&\leq -\pard{}{t} (\phi-Z)+g^{ij}\left(\pard{}{x^i}+\pard{}{y^i}\right)\left(\pard{}{x^j}+\pard{}{y^j}\right)(\phi-Z)\\
&=-\ho \phi +Z|A|^2 + \frac 4 d \delta^{ij}\pard{(\phi-Z)}{x^j}\ip{w}{\partial^x_i-\partial^y_i} \\
&\qquad +\frac{4}{d^2}\left(Z\delta^{jp}-h^{xpj}\right)[\delta_{jp}-\ip{\partial^y_j}{\partial^x_p}+2\ip{w}{\partial^x_j}\ip{w}{\partial^y_p-\partial^x_p}]\ .
\end{flalign*}
We now estimate the last term - we note that at this point all the above calculations are valid regardless of choice of parametrisation (and therefore $\partial^x_i$, $\partial^y_i$) and we now choose this to minimise the last term. We know that $Z>\kappa_{\text{max}}$ (as otherwise we are in the first case), and as in \cite[Lemma 6]{AndrewsLangfordMcCoy}, we now pick our coordinates so that we may see that the matrix in square brackets is the claimed nonpositive gradient term:

If $T_x M \perp w$, then using the above relations (and the fact that $\nu(x,t)-Zdw$ is a unit vector) then $T_y M \perp w$. In this case we pick $\partial_i^x=\partial^y_i$ for an orthonormal basis $\{\partial_i^x\}_{1\leq i \leq n}$ of $T_xM_t$. In this case, both the square bracket and $\pard{\phi}{x^i}=0$ and so the claimed equation holds. The same holds if $T_y M \perp w$.

Now we suppose that the projections of $w$ onto both $T_yM$ and $T_xM$ are nonzero. By choosing an orthonormal basis $\partial_1^x, \ldots, \partial_{n-1}^x$ spanning $w^\perp \cap T_x M_t$ and $\partial_i^y=\partial_i^x$ for $1\leq i \leq n-1$, we see that we only need to deal with $j=p=n$ and we may choose that 
\begin{align*}
\partial_n^x&=\frac{w-\ip{w}{\nu(x,t)}\nu(x,t)}{\sqrt{1-\ip{w}{\nu(x,t)}^2}}\\
\partial_n^y&=-\frac{w-\ip{w}{\nu(y,t)}\nu(y,t)}{\sqrt{1-\ip{w}{\nu(y,t)}^2}}=-\frac{(1-2\ip{\nu(x,t)}{w}^2)w+\ip{w}{\nu(x,t)}\nu(x,t)}{\sqrt{1-\ip{w}{\nu(x,t)}^2}}
\end{align*}
where we used \eqref{eq:nuy}. Writing $a=\ip{w}{\nu(x,t)}$, so that
\[\ip{\partial^y_n}{\partial^x_n} = \frac{-(1-2a^2)-a^2+a^2(1-2a^2)+a^2}{1-a^2} = 2a^2-1, \qquad \ip{\partial^x_i}{w}=\sqrt{1-a^2}=-\ip{\partial^y_i}{w}\]
\begin{flalign*}
1&-\ip{\partial^y_n}{\partial^x_n}+2\ip{w}{\partial^x_n}\ip{w}{\partial^y_n-\partial^x_n}=1+1-2a^2-4(1-a^2)=-2(1-a^2)<0
\end{flalign*}
and so we see that
\[\frac{4}{d^2}\left(Z\delta^{jp}-h^{xpj}\right)[\delta_{jp}-\ip{\partial^y_j}{\partial^x_p}+2\ip{w}{\partial^x_j}\ip{w}{\partial^y_p-\partial^x_p}] = -\frac{8}{d^2}(Z|w^\top|^2-h(w^\top,w^\top))\ .\]
As $Z>\kappa_\text{max}$, we have that
\[\frac{4}{d^2}(Z|w^\top|^2-h(w^\top,w^\top)) = -\frac{2}{d^2} \pard{\phi}{x_n} \ip{\partial_n}{w} = \sum_{i=1}^n\frac{\left(\n_{e_i}\phi\right)^2}{\phi-\kappa_i}\]
and so
\[\ho \phi \leq |A|^2 \phi-2\sum_{i=1}^n\frac{\left(\n_i\phi\right)^2}{\phi-\kappa_i}\ .\]
as required.
\end{proof}
\begin{cor}\label{cor:innernoncollapsed}
If $M_0$ is initially Alexandrov immersed and strictly mean convex, then there exists an $\alpha>0$ such that for all time, $M_t$ is  inner $\alpha$-noncollapsed.
\end{cor}
\begin{proof}
We have that $\frac{\ov{Z}}{H}$ is a viscosity subsolution of the heat equation on $M_t$. By compactness, and Lemma \ref{lem:continuity} there exists an $0<\alpha$ such that $\frac{\ov{Z}}{H}(x,0)\leq \alpha^{-1}$ for all $x\in M^n$. As a result of standard viscosity solution/barrier solution comparison theorems (see e.g. \cite[Proof of Lemma 9.8, p290]{AndrewsChowGuentherLangford} or \cite[Theorem 3.3]{CrandallIshiiLions}), this is preserved for all time, which due to Lemma \ref{lem:ZtoIntSphere} is equivalent to $M_t$ being inner $\alpha$-noncollapsed. 
\end{proof}

\subsection{Outer noncollapsing estimates via viscosity solutions}

We now repeat the above argument, but for outer noncollapsing. In this section we will take $\ov{\Psi}$, $\ov{\chi}$ as in definitions \ref{def:AlexComp} and \ref{def:outernoncollapsed}.

We now define the \emph{outer visible set} at $x\in M^n$ at time $t$ (w.r.t.\ $\ov{\Psi}_t$) as follows:
\[W(x,t):=\{y\in \ov{\chi}: \exists \text{ a geodesic (w.r.t.\ $\ov{g}$), $\gamma:[0,1]\ra \ov{\chi}$, with $\gamma(0)=x, \ \gamma(1)=y$}\}\ .\]
Similarly to the previous section, we will consider
\[\un{Z}(x,t):=\inf\{ Z(x,y,t): y\in \partial W(x,t)\cap \partial \ov{\chi},\  y\neq x\}\ ,\]
where here we are extending the definition of $Z(x,y,t)$ to allow $y\in \partial \ov{\Psi}$.

We now consider the function $Z(x,y,t)$ as in \eqref{eq:Zdef}. As the normal is now pointing into $\ov{\chi}_t$, we may replace Lemma \ref{lem:ZtoIntSphere} by the statement that there exists an exterior sphere in $\ov{\chi}$ at $x$ with principal curvature $\kappa$ iff
\[\inf_{y\in \partial W(x,t)\setminus \{x\}} Z(x,y,t)\geq -\kappa\ .\]
As the proof is identical to Lemma \ref{lem:ZtoIntSphere}, we do not include it here. 

As previously, we split the boundary of $W(x,t)$. For $x\in M^n$ and $y\in V(x,t)$, let $\gamma_{x,y}:[0,1]\ra V(x,t)$ be the unit speed geodesic starting at $x$ and going to $y=\gamma_{x,y}(1)$. Let
\[\partial W_\text{Reg}(x,t) := \{y\in \partial V(x,t): \text{$\gamma_{x,y}$ hits $M_t \cup N_t$ for the first time at $y$ and } \gamma'_{x,y}(1)\notin T(M_t\cup N_t)\}\ , \]
and
\[\partial W_\text{Sing}(x,t) := \partial W(x,t)\setminus \partial W_\text{Reg}(x,t)\ . \]

We now observe the equivalent of Lemma \ref{lem:Regular}
\begin{lemma}
Suppose that $x\in M^n$ and $y\in M^n \cup N^n$ are such that 
\[\min_{y'\in \partial W(x,t)\cap\partial\Omega} Z(x,y',t)=Z(x,y,t)\leq 0\ .\]
Then $y\in \partial W_{\text{Reg}}(x,t)$ or $Z(x,y,t)=0$.
\end{lemma}
\begin{proof}
This is identical to the proof of Lemma \ref{lem:Regular}.
\end{proof}

\begin{proposition}\label{prop:outerviscosity}
Suppose that $N_t$ satisfies $\ip{\ddt{Y}}{\mu} \geq -H^N $. Then, while $\un{Z}\leq 0$, $\un{Z}$ satisfies
\[\ho \un{Z}\geq |A|^2\un{Z}+2\sum_{\kappa_i>\un{Z}} \frac{\left(\n_i \un{Z}\right)^2}{\kappa_i-\un{Z}}\]
in the viscosity sense.
\end{proposition}
\begin{proof}
At an arbitrary point $x_0\in M_{t_0}$, we have three possibilities for realising $\un{Z}$:
\begin{enumerate}
\item There is a sequence $y_i\in M_t$ such that $y_i\ra x$ and $\un{Z}(x,t)=\lim_{i\ra \infty} Z(x,y_i,t)$
\item There exists a $y\in (\partial W(x,t)\cap M_t)\setminus\{x\}$ such that $\un{Z}(x,t)=Z(x,y,t)$.
\item There exists a $y\in \partial W(x,t)\cap N_t$ such that $\un{Z}(x,t)=Z(x,y,t)$.
\end{enumerate}
This time, we assume that we have a $C^2$ function $\phi(x,t)$ such that in a neighbourhood of $(x_0,t_0)$, for $t\leq t_0$, $\un{Z}(x,t)\geq\phi(x,t)$ with equality at $(x_0,t_0)$. 

In the first two cases above we may follow exactly the proof of Proposition \ref{prop:innerviscosity} with the inequalities the other way around. Correspondingly, we only need consider the final case. In this case we know that for all $y\in N_t \cap \partial W(x,t)$,
\begin{equation}
Z(x,y,t)\geq\un{Z}(x,t)\geq \phi(x,t)\ .\label{eq:onNt}
\end{equation}
The main difference here is the sign of the normal. This time, $\nu(x_0,t_0)$ points into $\ov{\chi}$ while $\nu(y_0,t_0)$ points out of $\ov{\chi}$. Recalling that $Z$ is negative, we have that the center of the inscribed ball is given by 
\[X(x,t) -Z^{-1}\nu(x_0,t_0) = Y(y_0,t_0)+Z^{-1}\mu(y,t)\]
so
\[\mu(y_0,t_0) = Zdw-\nu(x_0,t_0)\ .\]

Now calculating at $x_0,y_0,t_0$ in local orthonormal coordinates,
\begin{align*}
0=\pard{}{y^i}(\phi-Z) &= \frac{2}{d^2}\ip{\partial_i^y}{\nu(x_0,t_0) - d Z w}\\
0=\pard{}{x^i}(\phi-Z)&= \pard{\phi}{x^i} - \frac 2 d (h^{xp}_i\ip{w}{\partial^x_p}-Z\ip{w}{\partial^x_i})\ .
\end{align*}
Using these identities we have that
\begin{align*}
\frac{\partial^2}{\partial y^i\partial y^j}(\phi-Z)&=\frac 2{d^2}(Z\delta_{ij}+{}^N\!h^y_{ij})
\end{align*}
where the sign change is from the change in normal direction. The other second derivatives remain unchanged:
\begin{align*}
\frac{\partial^2}{\partial x^j\partial y^i}(\phi-Z)&=-\frac{2}{d^2}\left(Z\delta_j^p-h_j^{xp}\right)\ip{\partial^y_i}{\partial^x_p}- \frac 2 d \pard{\phi}{x^j}\ip{w}{\partial^y_i}\\
\frac{\partial^2}{\partial x^i\partial x^j}(\phi-Z)&=\frac{2}{d^2}(Z\delta_{ij}-h^x_{ij}) + Zh^x_{jp}\delta^{pq}h^x_{qi} - \frac 2 d\n_p h^x_{ij}\delta^{pq}\ip{w}{\partial^x_q}-Z^2 h^x_{ij} \\
&\qquad+ \frac 2 d \pard{\phi}{x^j}\ip{w}{\partial^x_i}+ \frac 2 d \pard{\phi}{x^i}\ip{w}{\partial^x_j}+\frac{\partial^2 \phi}{\partial x^i \partial x^j}
\end{align*}
Meanwhile, for the time derivative
\begin{align*}
\pard{}{t}(\phi-Z) &=\pard{\phi}{t}+\frac{2H(x)}{d^2} + \frac{2}{d^2}\ip{\ddt{Y}(y,t)}{\nu(x_0,t_0)-Zdw} - \frac{2}{d}\ip{w}{\n H(x)}-Z^2H(x)\\
&=\pard{\phi}{t}+\frac{2H(x)}{d^2} - \frac{2}{d^2}\ip{\ddt{Y}(y,t)}{\mu(y_0,t_0)} - \frac{2}{d}\ip{w}{\n H(x)}-Z^2H(x)
\end{align*}

As a result, as $(x_0, y_0, t_0)$ is a maximum of $\phi-Z$, we have that
\begin{flalign*}
0&\geq -\pard{}{t} (\phi-Z)+g^{ij}\left(\pard{}{x^i}+\pard{}{y^i}\right)\left(\pard{}{x^j}+\pard{}{y^j}\right)(\phi-Z)\\
&=-\ho \phi +Z|A|^2 +\frac{2}{d^2} \left[\ip{\ddt{Y}(y_0,t_0)}{\mu(y_0,t_0)}+H^N\right]\\
&\qquad + \frac 4 d \delta^{ij}\pard{(\phi-Z)}{x^j}\ip{w}{\partial^x_i-\partial^y_i} \\
&\qquad +\frac{4}{d^2}\left(Z\delta^{jp}-h^{xpj}\right)[\delta_{jp}-\ip{\partial^y_j}{\partial^x_p}+2\ip{w}{\partial^x_j}\ip{w}{\partial^y_p-\partial^x_p}]
\end{flalign*}
As this time we may assume that $Z(x,y,t)<\kappa_{\text{min}}(x,t)$ (where $\kappa_{\text{min}}$ is the smallest principal curvature at $x\in M_t$), we may again obtain a positive gradient term from the final term above. We now check this.
This time we have that for $j=p=n$ and we may choose that 
\begin{align*}
\partial_n^x&=\frac{w-\ip{w}{\nu(x,t)}\nu(x,t)}{\sqrt{1-\ip{w}{\nu(x,t)}^2}}&
\partial_n^y&
=-\frac{(1-2\ip{\nu(x,t)}{w}^2)w+\ip{w}{\nu(x,t)}\nu(x,t)}{\sqrt{1-\ip{w}{\nu(x,t)}^2}}\ .
\end{align*}
Writing $a=\ip{w}{\nu(x,t)}$,
\[\ip{\partial^x_n}{\partial^y_n} = 2a^2-1, \qquad \ip{\partial_n^x}{w}=\sqrt{1-a^2} = -\ip{\partial_n^y}{w}\]
\begin{flalign*}
1&-\ip{\partial^y_n}{\partial^x_n}+2\ip{w}{\partial^x_n}\ip{w}{\partial^y_n-\partial^x_n}=-2(1-a^2)\ .
\end{flalign*}
The result now follows as before as we may see that
\begin{align*}
\frac{4}{d^2}\left(Z\delta^{jp}-h^{xpj}\right)[\delta_{jp}-\ip{\partial^y_j}{\partial^x_p}+2\ip{w}{\partial^x_j}\ip{w}{\partial^y_p-\partial^x_p}]&=\frac{8}{d^2}(h(w^\top,w^\top)-Z|w^\top|^2)\\
&= 2\sum_i\frac{ \left(\n_{e_i} \phi \right)^2}{\kappa_i-\phi} \ .
\end{align*}
\end{proof}
As in Lemma \ref{lem:continuity}, while $\un{Z}\leq 0$, $\un{Z}$ is continuous. We therefore have the following.
\begin{cor}\label{cor:outernoncollapsed}
If $M_0$ is initially Alexandrov immersed and strictly mean convex and $N_t$ is an Alexandrov comparison solution. Then there exists an $\alpha>0$ such that for all time, $M_t$ is outer $\alpha$-noncollapsed.
\end{cor}
\begin{proof}
This follows as in the proof of Corollary \ref{cor:innernoncollapsed}.
\end{proof}
\section{Applications}
We can now extend existing two-dimensional mean curvature flow with surgery results to the Alexandrov immersed setting. 

We may do this quickly and efficiently as, in Theorem \ref{thm:noncollapsing}, we showed that there was a larger stationary comparison domain $\ov{\Psi}$ with a flat metric such that we may consider $M_t$ as an embedded flow, as in Lemma \ref{lem:Comparison}. Indeed, by mean convexity, writing $\ov{\Omega}_t\subset\ov{\Psi}$ as the moving domain inside $\ov{G}^*M_t$, we have $\ov{\Omega}_t \subset \ov{\Omega}_0$ by Lemma \ref{lem:Comparison}. We define
\begin{equation}r_0:=\sup \left\{ r\in \bb{R}_+: \forall x\in \Omega_0, \ \ov{B}^{\ov{g}}_r(x)\subset \overset{\circ}{\ov{\Psi}}\right\}\ .\label{eq:r0def}\end{equation}
We have that $r_0>0$ depends only on the initial data, and we also have that any $\ov{\Psi}$ ball of radius $r_0$ is contained in $\ov{\Psi}$, may be identified (via $\ov{G}$) as a subset of $\bb{R}^{n+1}$ and $M_t\cap \ov{B}_{r_0}(x)$ is an embedded, noncollapsed, mean convex flow with an interior determined by $\Omega_t$. Therefore, using the terminology of Haslhofer--Kleiner \cite{HaslhoferKleiner}, we locally have an $\alpha$-Andrews flow in a parabolic cylinder about the singular point . As a result, the local regularity theory in \cite{HaslhoferKleiner} also holds here, with the minor addendum that we also need the parabolic cylinders to be sufficiently small depending only on $r_0$.

\begin{theorem}[Gradient estimate for Alexandrov immersed flows]\label{thm:gradest}
Suppose that $M_0$ is compact, Alexandrov immersed and strictly mean convex and let $r_0$ be as in \eqref{eq:r0def}. Then there exists $\rho>0$ and $C_l>0$ depending only on $M_0$ such that: Given any $0<r<r_0$ and any  $p\in M_t$ such that $B_r(p)\subset \ov{\Omega}_{t'}$ for some $t'\in[0,T)$ and $H(p,t)<r^{-1}$, then 
\[\sup_{B_r(p)}|\n^l A|\leq C_l r^{-(l+1)}\]
for all $T-r^2<t\leq T$.
\end{theorem}
\begin{proof}
As in the above discussion, on $P(p,t,r)$ the flow is an $\alpha$-Andrews flow and the above is just a rewrite of \cite[Theorem 1.8']{HaslhoferKleiner}. 
\end{proof}
\begin{remark}
As in \cite[Remark 2.3]{HaslhoferKleiner} we note that the condition that $B_r(p)\subset \ov{\Omega}_{t'}$ and may be assumed at every flow a blowup sequence after discarding a finite number of terms.
\end{remark}

Next we need the non--collapsedness improvement of Brendle \cite{BrendleInscribedRadius}. 

\begin{theorem}\label{thm:inscrradest}
Suppose that $M_0$ is compact, Alexandrov immersed and strictly mean convex. Then for any $\delta>0$ there exists a $\sigma(\delta), C(\delta)>0$ such that $\ov{Z}\leq (1+\delta)H+CH^{1-\sigma}$.
\end{theorem}
We will follow the original proof of Brendle \cite{BrendleInscribedRadius} and but we note that for the above statement alone we could also simply apply the elegant proof of Haslhofer and Kleiner \cite{HaslhoferKleinerShort} (as in the proof of Theorem \ref{thm:gradest}) to get the same statement. However, in Theorem \ref{thm:surgery} we need a proof which can be easily modified to flows with surgery (see also \cite[Comment before section 2]{BrendleInscribedRadius}), and so we briefly describe the integral estimate proof.
\begin{proof}
By an identical proof to in \cite[Proposition 2.1]{BrendleInscribedRadius} we see that $\ov{Z}$ is also Lipschitz continuous and semi-convex. We have already demonstrated \cite[Proposition 2.3]{BrendleInscribedRadius} in Proposition \ref{prop:innerviscosity} and so \cite[Corollary 6]{BrendleInscribedRadius} also holds. The auxiliary equation of \cite[Proposition 3.1, Corollary 3.2]{BrendleInscribedRadius} now follow in an identical manner (where we note that Huisken and Sinestrari's convexity estimates \cite{HuiskenSinestrariConvexity} hold for immersed surfaces, so we may still apply these estimates). The proof now follows entirely analogously via Stampacchia iteration.
\end{proof}

As mentioned previously for $n\geq 3$ Huisken--Sinestrari \cite{HuiskenSinestrariConvexity} demonstrated surgery techniques which are applicable to MCF. We now sketch the minor adjustments to Brendle and Huisken's MCF of surfaces with surgery.

\begin{theorem}[Alexandrov immersed MCF with surgery]\label{thm:surgery}
There exists a notion of Alexandrov immersed mean curvature flow with surgery as in \cite{BrendleHuisken}.
\end{theorem}
\begin{proof}[Proof sketch]
We recall that in \cite{BrendleHuisken} surgery occurs when the maximum mean curvature hits $H_3>H_2>>H_1$, while the surgery happens at a scale of approximately $H_1$ on a neck of length $L<L_0$ (after rescaling by mean curvature). The choices of the values of $H_1, H_2, H_3$ are the final parameter choices and $H_1$ may be made arbitrarily large (see \cite[bottom of page 624]{BrendleHuisken}). We additionally stipulate that 
\[H_1>\frac{10^6L_0\Theta}{r_0}\ ,\]
which is enough to ensure that any neck we will want to apply surgery to will be strictly contained in a ball of radius $r_0$, and therefore may be viewed as embedded in $\bb{R}^{n+1}$ locally. We may therefore apply Brendle--Huisken's surgeries, as described in \cite[Section 6]{BrendleHuisken}. Furthermore, we will see below that in all proofs where embeddedness methods are used, we can replace this with local embeddedness via the assumption on $H_1$.

Next, we observe that all auxiliary results in \cite[Section 2]{BrendleHuisken} still hold for an Alexandrov immersed flow given the above choice of $H_1$, using local embeddedness: 
\begin{itemize}
\item The pseudolocality result \cite[Theorem 2.2]{BrendleHuisken} holds for immersed surfaces.
\item The modified outward-minimising gradient estimate, \cite[Theorem 2.3]{BrendleHuisken} will only be applied on (rescalings of) balls of radius strictly less than $r_0$, and so we may simply apply this theorem unchanged. After replacing variations in $\bb{R}^{n+1}$ with variations in the stationary outer comparison domain $\ov{\Psi}$, the work of Head \cite[Lemma 5.2]{Head2Convex} still indicates that the outward minimising property holds and is preserved regardless of surgeries (in modifying this, we are using that necks are in Euclidean embeddable balls).
\item Properties of surgery \cite[Theorem 2.5]{BrendleHuisken} may still be applied unchanged, due to the above considerations.
\item The proofs of Propositions 2.7, 2.8 in \cite{BrendleHuisken} still hold by local embeddedness due to our choice of $H_1$. 
\item \cite[Proposition 2.9]{BrendleHuisken} then follows from Propositions 2.7, 2.8.
\item \cite[Propositions 2.10 - 2.12]{BrendleHuisken} do not need to be modified at all.
\item The inscribed radius estimate \cite[Proposition 2.13]{BrendleHuisken} is a modified version of Brendle's Stampaccia inscribed radius estimate. The exact same modification applied to Theorem \ref{thm:inscrradest} yields the same result.
\item The proofs of Neck Detection Lemmas \cite[Theorems 2.14, 2.15]{BrendleHuisken} involves a contradiction argument blowing up flows on small balls. By our choice of $H_1$ the flow is embedded on the balls and the proof still goes through.
\item The proof of Proposition 2.16 still goes through due to our choice of $H_1$ (note that in this proof, due to the remark after the statement of the pointwise derivative estimate \cite[Proposition 2.9]{BrendleHuisken}, $H(p_1,t_1)$ must be close to $H_1$ and so we are again in a region which we may view to be embedded in $\bb{R}^{n+1}$).
\item \cite[Proposition 2.17]{BrendleHuisken} holds for immersed surfaces with a gradient estimate.
\item  \cite[Proposition 2.18]{BrendleHuisken} is only required in the embedded setting, see below.
\end{itemize}

As a result of this, we now see that the proof of \cite[Proposition 3.1]{BrendleHuisken}(which is used to find an initial neck given sufficiently large curvature) still holds by applying previously verified auxiliary statements. Finally, almost all of the proof of the Neck Continuation Lemma \cite[Proposition 3.2]{BrendleHuisken} immediately follows, until the final contradiction to get a convex cap. Here we need to be careful in the application of \cite[Proposition 2.18]{BrendleHuisken} (which requires an embedded cylinder). However, after rescaling by $H=\frac{2H_1}{\T}$ we only need a short piece of embedded cylinder to obtain the contradiction, and so, by choosing a rescaled length shorter than $L$, the original cylinder was contained in a ball of radius $r_0$, and so may be viewed as embedded in $\bb{R}^{n+1}$. Applying \cite[Proposition 2.18]{BrendleHuisken} the contradiction argument goes through and the Neck Continuation Lemma still holds.

As a result, Brendle and Huisken's surgery procedure goes through as claimed.
\end{proof}
\begin{remark}
The authors expect that it is also possible to modify the alternative surgery procedure of Haslhofer and Kleiner \cite{HaslhoferKleinerSurgery} to the Alexandrov immersed setting.
\end{remark}

\bibliographystyle{plain}
\bibliography{LitMCF.bib} 

\end{document}